\documentclass[11pt]{article}
\usepackage{amsmath, amsthm, amssymb}
\allowdisplaybreaks
\usepackage{amscd,amssymb}
\usepackage[all]{xy}
\usepackage{color}
\usepackage{verbatim}
\usepackage{mathrsfs}
\usepackage{graphicx,epsfig}

\usepackage{hyperref}
\usepackage[T1]{fontenc}
\usepackage[ansinew]{inputenc}
\usepackage{float}
\usepackage{setspace}

\date{}

\let\oldproofname=\proofname
\renewcommand{\proofname}{\rm\bf{\oldproofname}}
\newcommand{\remove}[1]{}

\newtheorem{thm}{Theorem}[section]
\newtheorem{cor}[thm]{Corollary}
\newtheorem{lem}[thm]{Lemma}
 \newtheorem{prop}[thm]{Proposition}
  
\theoremstyle{definition}
 
 \newtheorem{rmk}[thm]{Remark}

 \newtheorem{claim}[thm]{Claim}

 \def\1{\mathbf{1}}
\newcommand{\bR}{\mathbb{R}}
\newcommand{\bE}{\mathbb{E}}
\newcommand{\bP}{\mathbb{P}}

\newcommand{\N}{\mathcal{N}}

\newcommand{\no}{\nonumber}
\DeclareMathOperator{\lk}{{\textit{lk}}}
\begin{document}
\title{Spectral gap bounds for the simplicial Laplacian and an application to random complexes}

%\author{ Samir Shukla, }
\author{ Samir Shukla\footnote{Department of Mathematics, Indian Institute of Technology Bombay, India. samirshukla43@gmail.com}, D. Yogeshwaran\footnote{Statistics and Mathematics Unit, Indian Statistical Institute Bangaluru, India. d.yogesh@isibang.ac.in}}
\maketitle
%\maketitle

\begin{abstract}
	In this article, we derive two spectral gap bounds for the reduced Laplacian of a general simplicial complex. Our two bounds are proven by comparing a simplicial complex in two different ways with a larger complex and with the corresponding clique complex respectively. Both of these bounds generalize the result of Aharoni et al. (2005) \cite{ABM} which is valid only for clique complexes. As an application, we investigate the thresholds for vanishing of cohomology of the neighborhood complex of the Erd\"{o}s-R\'enyi random graph. We improve the upper bound derived in Kahle (2007) \cite{kahle} by a logarithmic factor using our spectral gap bounds and we also improve the lower bound via finer probabilistic estimates than those in Kahle (2007) \cite{kahle}. 
\end{abstract}


\vspace{0.1cm}
\section{Introduction}

Let $G$ be  a graph with vertex set $V(G)$ (often to be abbreviated as $V$) and let $L(G)$ denote the (unnormalized) {\it Laplacian} of $G$. Let  $0 = \lambda_1(G) \leq \lambda_2(G) \leq  \ldots \leq\lambda_{|V|}(G)$ denote the eigenvalues of $L(G)$ in ascending order. Here, the second smallest eigenvalue $\lambda_2(G)$ is called the {\it spectral gap}. The {\it clique complex} of a graph $G$ is the simplicial complex whose simplices are all subsets $\sigma \subset V$ which spans a complete subgraph of $G$. We shall denote the $k$th  reduced cohomology of a simplicial complex $X$ by $\widetilde{H}^k(X)$.  In this article, all simplicial complexes are assumed to be connected and we always consider the reduced cohomology with real coefficients. For more detailed definitions, see Section \ref{sec:prel}. 

In \cite{ABM}, Aharoni et al. proved that if the spectral gap of its $1$-skeleton is large enough, then the cohomology of the corresponding clique complex vanishes. 
\begin{thm} \cite[Theorem 1.2]{ABM} \label{t:abm}
 Let $X$ be the clique complex of a graph $G$. If $\lambda_2(G) > \frac{k|V|}{k+1}$, then $\widetilde{H}^k(X) = 0$.
 \end{thm}
Aharoni et al. (\cite{ABM}) used Theorem \ref{t:abm} to derive a lower bound for the homological connectivity of the independence complex of a graph $G$ (a simplicial complex whose simplices are the independent sets of $G$) and which implies Hall type theorems for systems of disjoint representatives in hypergraphs.

Theorem \ref{t:abm} can be viewed as a global counterpart for clique complexes of the spectral gap  results of Garland (\cite[Theorem 5.9]{G}) and Ballman-{\' S}wi\c{a}tkowski (\cite[Theorem 2.5]{BS})  for vanishing of cohomology of a simplicial complex. In their simplest form,  these results say that for a pure $k$-dimensional finite simplicial complex $\Delta$, if the spectral gap of the link $\lk_{\Delta}(\tau)$ is sufficiently large for every ($k-2$)-dimensional simplex $\tau$, then $\widetilde{H}^{k-1}(\Delta) = 0$. A very powerful application of the afore-mentioned result can be found in Kahle (\cite{MK1}) where this was used to derive sharp vanishing thresholds for cohomology of random clique complexes. See \cite{Hoffman2017} for more applications of this spectral gap result in random topology. Recently,  Hino and Kanazawa (\cite[Theorem 2.5]{HK}) generalized this result of  Garland and Ballman-{\' S}wi\c{a}tkowski, and  upper bounded the $(d-1)^{th}$ Betti number of a pure $d$-dimensional simplicial complex $\Delta$ by the sum (taken over all $(d-2)$-dimensional simplices $\tau$) of the number of `suitably small' eigenvalues in the spectrum of the laplacian of the link $\lk_{\Delta}(\tau)$. They used this quantitative version of the spectral gap result to prove weak laws for (persistent) lifetime sums of randomly weighted clique and $d$-dimensional complexes.

Motivated by such applications of spectral gap bounds to random complexes, we seek to generalize Theorem \ref{t:abm} to more general simplicial complexes. We achieve two different generalizations (see Corollaries \ref{cor:subcomplex} and \ref{cor:general}) by comparing an arbitrary simplicial complex with a larger complex and the corresponding clique complex in two different ways. Our aim in exploring this generalization was to obtain vanishing threshold for cohomology in other random complex models. We use one of our generalizations to improve the vanishing threshold for cohomology of a random neighborhood complex (see Theorem  \ref{t:main}) by a logarithmic factor. After the result, we also discuss why it is difficult to apply Garland's method and hence a different spectral gap result is needed. By computing the probabilities involved more precisely than in \cite{kahle}, we also improve the lower bound by a polynomial factor.

The paper is organised as follows. In Section \ref{subsec:notation}, we introduce some notation, which we shall use in rest of the paper.
In Section \ref{subsec:spectral}, we state our results, which relate the cohomology and spectral gap. In Section \ref{subsec:randomneighborhood}, we state the results about the neighborhood complexes of a random graph.  We also discuss our improvements in relation to the results of \cite{kahle}. We give the necessary preliminaries from graph theory and topology in Section \ref{sec:prel} . Section \ref{sec:proof} contains the proofs of the results stated in Sections \ref{subsec:spectral} and \ref{subsec:randomneighborhood}. 

\subsection{Notation} \label{subsec:notation}

%\blue{Note the change in the notation of 1-skeleton of $X$ }

We shall use the following notations throughout this paper.  Let $X$ be a (simplicial) complex on $n$ vertices.  We denote by $G_X$ the  $1$-skeleton of $X$, i.e., $G_X$ is the graph whose vertices are the $0$-dimensional simplices and the edges are the $1$-dimensional simplices of $X$. We shall always assume $G_X$ to be connected. Let $X(k)$ denote the set of all $k$-dimensional oriented simplices of $X$. $X$ is said to be a {\em clique complex} if for all $k \geq 0$,  $X(k)$ is the set of $(k+1)$-cliques in the graph $G_X$.  For $k \geq -1$, let $C^k(X; \bR)$ denote the space of real valued $k$-cochains of $X$. Let $\delta_k(X) : C^k(X;\bR) \to C^{k+1}(X;\bR)$ denote the coboundary operator. 

For $k \geq 0$, let $\delta_k^{\ast}(X)$ denote the adjoint of $\delta_k(X)$ and $\Delta_k(X) := \delta_{k-1}(X) \delta_{k-1}^{\ast}(X) + \delta_{k}^{\ast}(X)\delta_k(X)$ denote the (simplicial) Laplacian (see Section \ref{sec:prel} for details). Let $\mu_k(X)$ denote the minimal eigenvalue of $\Delta_k(X)$. Observe that $\lambda_2(G_X) = \mu_0(X)$. We again emphasize that we consider reduced cohomology with real coefficients.

We shall now define two ways to measure the difference between two complexes. The first compares a complex to its subcomplex whereas the second compares a complex $X$ to the corresponding clique complex of $G_X$. For  $k \geq 1$, define
\begin{equation}
\label{e:Sk}
S_k(X,X') := \max \limits_{\sigma \in X'(k)} | \{ \tau \in X(k+1) \setminus X'(k+1) \ | \  \sigma \subset \tau  \}|, 
\end{equation}
where $X'$ is a subcomplex of $X$.  We shall now on use $X$ to denote a complex and $X'$ to denote a subcomplex of $X$.

For a simplex $\sigma \in X$, the {\it link} of $\sigma$ is the complex defined as 
$$
\lk_X(\sigma) := \{\eta \in X  \ | \ \sigma \cup \eta \in X \ \text{and} \ \sigma \cap \eta = \emptyset\}.
$$ 
For $k \geq 1$ and $1 \leq j \leq k+1$, define
\begin{align}
\label{e:Dkj}
D_k(X,j) := & \max\limits_{\sigma \in X(k)} | \{u  \ | \ u \notin \lk_X(\sigma) \ \text{and} \ \exists \ \mbox{exactly} \ j \ \mbox{vertices} \ v_1, \ldots, v_j \in \sigma \ \mbox{such that} \nonumber\\ 
& \, \, u \in \lk_X(\sigma \setminus \{v_i\}) \ \forall \ 1\leq i \leq j \}|.
\end{align}

\begin{rmk}
\label{r:Dk}
If the $k$-skeleton of $X$ is a clique complex of $G_X$ and  $u \in \lk_X(\sigma \setminus \{v\}) \cap \lk_X(\sigma \setminus \{w\})$ for some $\{v, w\} \subseteq \sigma$ then $u \in  \lk_X(\sigma \setminus \{v\}) \ \forall \ v \in \sigma$, {\it i.e.}, any $(k+1)$-subset of $\sigma \cup \{u\}$ will be a $k$-simplex. Therefore, in this case $D_k(X, j) = 0$ for all $2 \leq j \leq k$ and  
\begin{align}
\label{def:Dk}
D_k(X,k+1)= &\max\limits_{\sigma \in X(k)} | \{ w \ |  \ w \notin \lk_X(\sigma) \ \text{and any} \ (k+1)\text{-subset of} \ \sigma \cup \{w\} \ \text{is a} \ k\text{-simplex}  \}| . 
\end{align}

Thus, if $X$ is a clique complex then $D_k(X, j) = 0$ for all $2 \leq j \leq k+1$.    
\end{rmk}

\subsection{Spectral  gap and cohomology}
\label{subsec:spectral}

We shall now present our two spectral gap results and corollaries that generalize Theorem \ref{t:abm}. We first recall the following key theorem from \cite{ABM} that was used to derive Theorem \ref{t:abm}. 
\begin{thm}\cite[Theorem $1.1$]{ABM} \label{t:abm1}
Let $X$ be a clique complex. For $k \geq 1$,	
$$ k \mu_k(X) \geq (k+1) \mu_{k-1}(X) - |V(G_X)|. $$
\end{thm}

We prove our first main spectral gap result by directly comparing the operators $\delta_k(X)$ and $\delta_k(X')$. Following the theorem, we state a simple corollary which generalizes Theorem \ref{t:abm}.. 
\begin{thm} \label{t:subcomplex}
   For every simplicial complex $X$ and every   subcomplex $X'$ of $X$, and for  $k \geq 1$,
\begin{equation} \label{e:subcomplexname}
\mu_k(X') \geq \mu_k(X) -(k+2)S_k(X,X').
\end{equation}
\end{thm}
\begin{cor} \label{cor:subcomplex}
Let $X$ be a clique complex and $X'$ have the same $1$-skeleton as $X$ i.e.,$G_{X'} = G_X$.  If $\lambda_2(G_X)> \frac{kn}{k+1} + \frac{k+2}{k+1} S_k(X,X')$, then $\widetilde{H}^k(X') = 0$.
\end{cor} 
If $X'= X$, then $S_k(X,X') = 0$ and so Corollary \ref{cor:subcomplex} implies Theorem \ref{t:abm}. Now, we present our generalization of Theorem \ref{t:abm1} using $D_k(X,j)$'s and another corollary that generalizes Theorem \ref{t:abm}.
\begin{thm} \label{t:general}
For any simplicial complex $X$ and for $k \geq 1$,
\begin{align} 
 k \mu_k(X) \geq (k+1) \mu_{k-1}(X) - n - \sum\limits_{j=2}^{k+1} (k(k+1)+j) D_k(X,j).
\end{align}
\end{thm}
%
%Hence,  as a special case of Theorem \ref{general} we obtain the following.
% \begin{thm} \label{kclique}
%Let $k$-skeleton of $X$ be the clique complex of $G_X$. For $k \geq 1$,
% \begin{align} 
% k \mu_k(X) \geq (k+1) \mu_{k-1}(X) - n - (k(k+1)+1) D_k(X,k).
% \end{align}
%  \end{thm}
\begin{cor} \label{cor:general}
    Let $k$-skeleton of $X$ be same as that of the $k$-skeleton of the clique complex of $G_X$. If $\lambda_2(G_X) > \frac{kn}{k+1} + (k+ 1) D_k(X,k+1)$, then $\widetilde{H}^k(X) = 0$.
 
\end{cor}
From Remark \ref{r:Dk}, we know that for a clique complex $X$, $D_k(X, j) = 0 \  \forall \ 2 \leq j \leq k+1 $. Hence, in this case Theorem \ref{t:general} implies Theorem  \ref{t:abm1} and Corollary \ref{cor:general} implies Theorem \ref{t:abm}. The proof of Theorem \ref{t:general} follows the ideas of \cite{ABM} but some of the terms that cancel out in the case of clique complexes do not cancel out for a general simplicial complex. Hence, one requires more care in deriving the necessary bounds. 
%
%As an application of Theorem \ref{generalcor}, we prove Theorem \ref{main}.
%\blue{In the following theorem also, I don't think we need to define $X'$}
% Here, $S_k(X, X)'s$ are  a quantitative  way of measuring the difference %between $X$ and its subcomplex $X'$.

\subsection{Neighborhood complex of a random graph} \label{subsec:randomneighborhood}
We shall now introduce neighborhood complex of random graphs, recall results from \cite{kahle} and state our results about the vanishing of cohomology of these random complexes. For more on random graphs, we refer the reader to \cite{Janson11, Frieze16} and for a survey on random simplicial complexes refer to \cite{Kahle14a}. 

The {\it neighborhood complex}, $\N(G)$ of a graph $G$ is the simplicial complex whose  simplices are those subsets $\sigma$ of $V$ which have a common neighbor. Neighborhood complexes were introduced by Lov{\' a}sz (\cite{lovasz}) in his proof of the famous Kneser conjecture. We now introduce the Erd\"{o}s-R\'enyi random graph $G(n,p)$ on $n$ vertices and with edge-probability $p$. $G(n,p)$ is constructed by deleting edges of the complete graph on $n$ vertices independently of each other with probability $1-p$ or equivalently the edges are retained independently of each other with probability $p$. In this article, we consider $p$ as a function of $n$. A {\em graph property $\mathcal{P}$} is a class of graphs such that for any two isomorphic graphs either both belong to the class or both do not belong to the class. For any graph property $\mathcal{P}$, we say that $G(n, p) \in \mathcal{P}$ {\it with high probability} ({\it w.h.p.}) if $\mathbb{P} (G(n, p) \in \mathcal{P}) \to 1$ as $n \to \infty$. We shall also say $\mathcal{P}$ holds for $G(n,p)$ instead of   $G(n, p) \in \mathcal{P}$. %,that if we split the $n$-subsets of a $(2n+k)$-element set into $k+1$ classes, then  one of the classes will contain two disjoint $n$-subsets. 
%The topological connectivity of $\N(G)$ is useful in finding  a lower bound for the chromatic number of $G$ (see \cite{lovasz}).

In \cite{kahle}, M. Kahle considered the  neighborhood complex of the Erd\"{o}s-R\'enyi random graph. He showed that (see \cite[Theorem 2.1]{kahle}), if ${n \choose k+2} (1-p^{k+2})^{n-(k+2)} = o(1)$, then w.h.p. $\widetilde{H}^i(\N(G(n,p))) = 0$, for $i \leq k$. In particular, if $p = \Big(\frac{(k+2)\log n  +c_n}{n}\Big)^{\frac{1}{k+2}}$ with $c_n \to \infty$, then w.h.p. $\widetilde{H}^i(\mathcal{N}(G(n, p))) = 0$ for $i \leq k$.  Using Corollary \ref{cor:general}, we achieve the following improvement on Kahle's result.
\begin{thm}
\label{t:main}
Let $k \geq 1$. If $p = \Big(\frac{(k+1)\log n  +c_n}{n}\Big)^{\frac{1}{k+2}}$ with $c_n \to \infty$, then  $\widetilde{H}^i(\mathcal{N}(G(n, p))) = 0$ w.h.p. for $i \leq k$.
\end{thm}
Note that in Theorem \ref{t:main},  ${n \choose k+2} (1-p^{k+2})^{n-(k+2)} \to \infty$ and therefore we cannot apply Kahle's result in this case. His proof involves showing that for $p$ satisfying ${n \choose k+2} (1-p^{k+2})^{n-(k+2)} = o(1)$, $\N(G(n, p))$ has the full $(k+1)$-skeleton, i.e., any $t$-tuple of vertices form a $(t-1)$-simplex in $\N(G(n,p))$ for $t \leq k+2$. This trivially yields that $\widetilde{H}^i(\mathcal{N}(G(n, p))) = 0$ for $i \leq k$. But one would expect that this is a very strong condition for vanishing of cohomology and our theorem shows that this can be reduced a little. We expect that our bound for the threshold for vanishing of cohomology to be reduced even further.

The above theorem is one of our motivations to prove spectral bounds for vanishing of cohomology for general complexes. This study was inspired by the proof of a sharp threshold result for vanishing of cohomology of clique complexes of Erd\"{o}s-R\'enyi random graphs in \cite{MK1} which was proven using the spectral gap result of Garland and Ballman-{\' S}wi\c{a}tkowski. This required to show that with suitably high probability the spectral gap is sufficiently large for the normalized Laplacian of the $1$-skeleton of the link. For a clique complex of an Erd\"{o}s-R\'enyi random graph, it is easy to see that the $1$-skeleton of the link of a simplex is also an Erd\"{o}s-R\'enyi random graph and hence by proving suitable spectral bounds for the normalized Laplacian of Erd\"{o}s-R\'enyi random graphs (\cite{Hoffman2017}), the result of Garland and Ballman-{\' S}wi\c{a}tkowski was applied. But it is not easy to use the same argument to prove Theorem \ref{t:main}, as the $1$-skeleton of the link of a simplex in the neighborhood complex of the  Erd\"{o}s-R\'enyi random graph is not an  Erd\"{o}s-R\'enyi random graph. It has a complicated dependency structure making it harder to analyse the spectral gap of the corresponding random graph. 

Kahle (see \cite[Corollary 2.9]{kahle}) also  showed that for $p= n^{\alpha}$, if $\frac{-2}{k+1} < \alpha <  \frac{-4}{3(k+1)}$, then w.h.p. $\widetilde{H}^k(\mathcal{N}(G(n, p)))$ $ \neq 0$. We derive more exact bounds for the probabilities involved but still use the same argument as that of \cite{kahle} to extend this result as well.
\begin{prop} \label{p:extension}
 Let $p = n^{\alpha}$. If $\frac{-2}{k+1} < \alpha <  \frac{-1}{k+1}$, then w.h.p.   $\widetilde{H}^k(\mathcal{N}(G(n, p))) \neq 0$.
 \end{prop}
Despite the improvement of the bounds presented here, it is still an open problem to determine sharp bounds for vanishing of cohomology of neighborhood complexes. From Theorem \ref{t:main}, Proposition \ref{p:extension} and \cite[Corollary 2.5]{kahle}, we summarize the known bounds as follows : For $k \geq 1$, 
\begin{align*}
\widetilde{H}^k(\mathcal{N}(G(n, p))) & = 0 \ \ \mbox{w.h.p. if $p = n^{\alpha}$ with $\alpha < \frac{-4}{k+2}$ for $k$ even and $\alpha < \frac{-4(k+2)}{(k+1)(k+3)}$ for $k$ odd},\\
\widetilde{H}^k(\mathcal{N}(G(n, p))) & \neq 0 \ \ \mbox{w.h.p. if $p = n^{\alpha}$ with $\frac{-2}{k+1} < \alpha <  \frac{-1}{k+1}$}, \\
\widetilde{H}^k(\mathcal{N}(G(n, p))) & = 0 \ \ \mbox{w.h.p. if $p = \Big(\frac{(k+1)\log n  +c_n}{n}\Big)^{\frac{1}{k+2}}$ with $c_n \to \infty$}. 
\end{align*}
\begin{rmk}
	In \cite[Corollary 2.5]{kahle}, the homology groups are given with integer coefficients, but the result was proven by  showing that $\N(G(n,p))$ deformation retracts onto a subcomplex of dimension $k-1$. Hence, the same proof is valid irrespective of the coefficients of the homology groups. Also, we have used the fact that the homology and cohomology groups with real coefficients are isomorphic to each other.
	\end{rmk}

\section{Preliminaries} \label{sec:prel}
A  {\it graph} $G$ is a  pair $(V, E)$,  where $V$ is the set of vertices of $G$  and $E \subset V \times V$ is called the set of edges. We shall always assume $V$ to be finite. If we wish to avoid ambiguities about the graph under consideration, we will denote that vertex set as $V(G)$ and the edge set as $E(G)$. If $(u, v) \in E$, it is also denoted by $u \sim v$ and we say that $u$ is adjacent to $v$. For any  $A \subset V$, the neighborhood of $A$ is defined as $N(A):= \{u \in  V \ | \ u \sim a \,\,\forall\,\, a \in A \}$. The {\it degree} of a vertex $v$ is $|N(v)|$ and it is denoted by $deg(v)$. For a subset $X \subset V$, the induced subgraph $G[X]$ is the subgraph whose vertex set is $X$ and the set of edges is $\{(u, v) \in E \ | \ u, v \in X\}$. The {\it complete graph} or a {\it clique} of order $n$ is a graph on $n$ vertices, where any two distinct vertices are adjacent and it is denoted by $K_n$. All the graphs in this article are assumed to be undirected and simple (i.e., $(x, y) \in E$ implies $x \neq y$).

The (unnormalized) {\it Laplacian} of a graph $G$ is the $|V| \times |V|$ matrix $L(G)$  given by
$$
            L(G)(x, y) := \begin{cases}
            \ deg(x)&  x=y ,\\
\ -1 &  (x,y) \in E,\\

\ 0 &   \text{otherwise}.\\
                       \end{cases}
 $$
  For details about the graph Laplacian, we refer the reader to \cite{bapat}.
We next introduce the concept of simplicial complexes, which are higher dimensional counterparts of graphs.

A  finite (abstract) {\it simplicial complex X} is a family of subsets of a finite set, which is  closed under the deletion of elements, i.e.,   if $\alpha \in X$ and $\beta \subset \alpha$,
then $\beta \in X$. For $\sigma \in X$, the dimension of $\sigma$ is defined to be $|\sigma| - 1$ and denoted by $dim(\sigma)$. If dim($\sigma$) = $k$, then it is said to be a $k$-dimensional simplex or $k$-simplex. The $0$-dimensional
simplices are called vertices of $X$. We denote the set of vertices of $X$ by $V(X)$. The {\it boundary} of a $k$-dimensional simplex
$\sigma $ is the simplicial complex, consisting of  all  simplices $\tau \subset \sigma$ of dimension $\leq k-1$. We refer the reader to \cite{dk,Munkres} for more details about simplicial complexes.
%
%\begin{defn}
%The {\it neighborhood complex}, $\N(G)$ of a graph $G$ is the simplicial complex whose  simplices are those subsets $\sigma$ of $V(G)$ which have a common neighbor, i.e., $N(\sigma) \neq \emptyset$
%\end{defn}
%
Let $X$ be a simplicial complex. Two ordering of vertices of a simplex $\sigma = \{v_0, v_1, \ldots, v_k\}$ are called equivalent if they differ from one another  by an even permutation. Thus the ordering of the vertices of a simplex divided into two equivalences classes. Each of these classes is called an orientation of $\sigma$. An oriented simplex is a simplex $\sigma$ together with  an orientation and we denote it by $[ v_0, \ldots, v_k ]$.

Let  each simplex of $X$ have an  arbitrary but fixed orientation. Let $X(k)$ denote the set of oriented $k$-simplices of $X$.  For $k \geq 0$, let $C_k(X)$ denote the free abelian group with basis $X(k)$, with the  relation $[ v_0,v_1, \ldots, v_k ] =  - [ v_1, v_0, \ldots, v_k ]$ for each $k$-simplex $\sigma = \{v_0, \ldots, v_k\}$.

% \red{we can also remove the definition of boundary maps  $d_k$}
%For $k \geq 1$, the $k$-th boundary operator $d_k : C_k(X) \to C_{k-1}(X)$ is defined by 
%%
%$$ d_k [v_0, \ldots, v_k] :=  \sum\limits_{i=0}^{k} (-1)^i [v_0, \ldots, \hat{v_i} , \ldots, v_k]. $$
%%
%
%By setting $C_{-1}(X) = \mathbb{Z}$, define $d_0: C_0(X) \to \mathbb{Z}$ by $d_0(v) = 1$ for all $v \in X(0)$. It is well known that $d_{k-1} d_k = 0 $ for all $k \geq 1$. 

For $k \geq 0$, let $C^k(X; \bR)$ be the dual group Hom$(C_k(X); \bR)$ where Hom$(.,.)$ is the group of Homomorphisms. The elements of $C^k(X;\bR)$  are called $k$-cochains of $X$. For an ordered $(i+1)$-simplex $\sigma = [v_0, \ldots, v_{i+1}]$, the $j$-face of $\sigma$ is an ordered $i$-simplex $\sigma_j=[v_0, \ldots, \hat{v_j}, \ldots, v_{i+1}]$.  The co-boundary operator $\delta_k (X) : C^k(X;\bR) \to C^{k+1}(X;\bR)$ is given by
 $$ 
 \delta_k(X) \phi (\sigma) :=  \sum\limits_{j=0}^{k+1} (-1)^j \phi(\sigma_j).
 $$
 By letting $C^{-1}(X;\bR) = \bR$, define  $\delta_{-1}(X) :C^{-1}(X;\bR) \to C^{0}(X;\bR) $  by $\delta_{-1}(X)(x)(v) = x$ for all $x \in \bR$ and $v \in X(0)$.  It is well known that $\delta_{k} \delta_{k-1} = 0 $ for all $k \geq 1$. For $k \geq 0$, the quotient Ker $\delta_k(X)$ / Im $\delta_{k-1}(X)$ is called the $k$-th reduced cohomology group of $X$ with real coefficients and it is denoted by $\widetilde{H}^k(X)$. For more details about cohomology, we refer the reader to \cite{Munkres} and see \cite{Herbert} for a simpler introduction. 

For each $k \geq -1$, we can define the standard inner product on $C^k(X;\bR)$ by 
 $ \langle \phi, \psi \rangle := \sum_{\sigma \in X(k)} \phi(\sigma)\psi(\sigma)$
 and the corresponding  $L^2$ norm $|| \phi || := (\sum_{\sigma \in X(k)} \phi(\sigma)^2)^{\frac{1}{2}}$. 
 
 Let $\delta_k^{\ast}(X): C^{k+1}(X;\bR) \to C^{k}(X;\bR)$ denote the adjoint of $\delta_k(X)$ with respect to the standard  inner product, i.e., the unique operator satisfying $ \langle \delta_k(X)\phi, \psi \rangle = \langle \phi, \delta_k^{\ast}(X) \psi \rangle$ for all $\phi \in C^k(X; \bR)$ and $\psi \in C^{k+1}(X;\bR)$. The reduced $k$-Laplacian of $X$ is the mapping 
 $$
 \Delta_k (X) := \delta_{k-1}(X) \delta_{k-1}^{\ast}(X) + \delta_k^{\ast}(X)\delta_k(X) : C^{k}(X;\bR) \to C^{k}(X;\bR).
$$ 

It can be easily verified that if $\mathbb{I}$ denotes the $|V(G_X)| \times $ $|V(G_X)|$ matrix with all entries $1$, then $\mathbb{I}+ L(G_X)$ represents $\Delta_0(X)$ with respect to the standard basis. In particular, the minimal eigenvalue of $\Delta_0(X)$ ({\it i.e.}, $\mu_0(X)$) is  $\lambda_2(G_X)$. More details about the operator $\Delta_k(X)$ can be found in \cite{BS} and \cite{G}. 

We now recall the following well known simplicial Hodge theorem.
\begin{prop}\label{p:hodge}
For $k \geq 0$, Ker $\Delta_k(X) \cong \widetilde{H}^k(X)$.
\end{prop}

%\blue{You asked me to include max-min principle here. But I find it more suitable to introduce this result just before the proof of Theorem \ref{subcomplex} (see Proposition \ref{maxmin}) }

 \section{Proofs} \label{sec:proof}
\subsection{Proofs of the results in section \ref{subsec:spectral}} \label{sec:spectralgap}
 
Throughout  this article, for any positive integer $m$, we denote the set $\{1, \ldots, m\}$ by $[m]$. Recall that, $X$ is a complex and $X'$ is a subcomplex of $X$. For two oriented simplices   $\eta \in X$ and $\tau \in \lk_X(\eta)$, $\eta\tau$ denotes their oriented union, i.e., if $\eta = [v_0, \ldots, v_k]$ and $\tau = [u_0, \ldots, u_l]$, then $\eta\tau = [v_0, \ldots, v_k, u_0, \ldots, u_l]$.
  Further, for any $k$-cochain $\phi$ of $X'$, we also consider $\phi$ as a cochain of $X$ by simply taking $\phi(\sigma) = 0$  whenever $\sigma \in X(k) \setminus X'(k)$ and  a cochain $\psi$ of $X$ can be considered  as a cochain of $X'$ by taking the restriction of $\psi$ on $X'$. 
%  
%  or  $\phi \in C^k(X'; \bR)$, we define $\widehat{\phi} \in C^k(X, \bR)$ by 
%  
%  \begin{center}
%            $\widehat{\phi}(\sigma)= \begin{cases}
%            \phi(\sigma)  & \text{if}\  \sigma \in X' ,\\
%\ 0 & \text{if} \ \sigma \in X \setminus X',\\
%                       \end{cases}$
%  \end{center}
% and for $\psi \in C^k(X;\bR)$ let $\widecheck{\psi} \in C^k(X';\bR)$ denotes the restriction of $\psi$ on $X'$.

 In the rest of the section, we shall abbreviate as follows : $\delta_k = \delta_k(X), \delta'_k =\delta_k(X'), \delta_k^{\ast} = \delta_k^{\ast}(X),\delta_k^{'\ast} = \delta_k^{\ast}(X'), \Delta_k = \Delta_k(X)$ and $ \Delta_k' = \Delta_k'(X)$.
\begin{lem} \label{l:lemma1}
For $\phi \in C^{k}(X';\bR)$
\begin{equation} \label{e:equal1}
 ||\delta_{k-1}^{\ast}\phi||^2  = ||\delta_{k-1}^{'\ast}\phi||^2.
\end{equation}
\end{lem}
\begin{proof} By the definition of the adjoint and coboundary operator,   it can easily verified   that  for any $\tau \in X(k-1)$,
	% and $\sigma \in X'(k-1)$, by the definition of $\phi$ on $X$, we have that
 % 
 \begin{align*}
 \delta_{k-1}^{\ast} \phi(\tau) = \sum_{v \in \lk_{X}(\tau)} \phi(v\tau).
\end{align*}
From the definition of $\phi$ on $X$, we have that

%\begin{align*}
%\delta_{k-1}^{\ast} \phi(\tau) = \sum_{v \in lk_{X}(\tau)} \phi(v\tau) = \sum_{v \in lk_{X'}(\sigma)} \phi(v\sigma) = \delta_{k-1}^{'\ast} \phi(\sigma).
%\end{align*}

$$
\delta_{k-1}^{\ast} \phi(\sigma)= \begin{cases}
0  & \sigma \in X(k-1) \setminus X'(k-1),\\
\  \delta_{k-1}^{'\ast}\phi(\sigma)  & \sigma \in X'(k-1).\\
\end{cases}
$$

Thus, the result follows.
\end{proof}

We shall require the following simple inequality : For any real numbers $x_1, x_2, \ldots, x_n$, it holds that
\begin{align}
\sum\limits_{\{i, j\} , i \neq j} x_i x_j \leq \frac{(n-1)}{2} \sum\limits_{i=1}^{n} x_i^2.  \label{e:inequality}
\end{align}

  \begin{lem} \label{l:lemma2}
For $ \phi \in C^{k}(X';\bR)$, recalling $S_k(X,X')$ as defined in \eqref{e:Sk}, we have that 
\begin{equation} \label{e:lessequal1}
  ||\delta_k \phi||^2 - ||\delta'_k \phi||^2 \leq (k+2) S_k(X, X') ||\phi||^2.
  \end{equation}
 \end{lem}
  
\begin{proof}
%\begin{equation}
\begin{align} 
 ||\delta_k \phi||^2 - ||\delta'_k \phi||^2 & =  \sum_{\tau \in X(k+1) \setminus X'(k+1)} \Big( \delta_k \phi(\tau) \Big)^2 \nonumber\\
 & = \sum_{\tau \in X(k+1) \setminus X'(k+1)} \sum_{i=0}^{k+2} (-1)^i \phi(\tau_i) \sum_{j=0}^{k+2} (-1)^j \phi(\tau_j) \nonumber\\
 & = \sum_{\tau \in X(k+1) \setminus X'(k+1)} \Big(\sum_{i=0}^{k+2} (-1)^{2i} \phi(\tau_i)^2 +  \sum_{ i \neq j} (-1)^{i+j} \phi(\tau_i) \phi(\tau_j) \Big) \nonumber\\
 & = \sum_{\tau \in X(k+1) \setminus X'(k+1)} \Big(\sum_{i=0}^{k+2} \phi(\tau_i)^2 + 2 \sum_{ \{i, j \}, i \neq j } (-1)^{i+j} \phi(\tau_i)\phi(\tau_j)\Big) \nonumber \\
 & \leq \sum_{\tau \in X(k+1) \setminus X'(k+1)} \Big( \sum_{i=0}^{k+2} \phi(\tau_i)^2 + (k+1) \sum_{i=0}^{k+2} \phi(\tau_i)^2 \Big) \notag \label{middle}
 \end{align}
where last inequality follows from \eqref{e:inequality}. Hence, we derive that
\begin{align}
 ||\delta_k \phi||^2 - ||\delta'_k \phi||^2 & \leq (k+2) \sum_{\tau \in X(k+1) \setminus X'(k+1)} \sum_{i=0}^{k+2} \phi(\tau_i)^2 \nonumber\\
 & = (k+2)\sum_{\sigma \in X(k)} \phi(\sigma)^2| \{ \tau \in X(k+1) \setminus X'(k+1)  \ | \ \sigma \subset \tau\}| \nonumber\\
 & = (k+2)\sum_{\sigma \in X'(k)} \phi(\sigma)^2  | \{ \tau \in X(k+1) \setminus X'(k+1)  \ | \ \sigma \subset \tau\}| \nonumber\\
 & \leq (k+2) S_k(X, X') ||\phi||^2 \nonumber.
\end{align}
\end{proof}
%
%\blue{I have given the reference for minimax principle as a corollary and %an exercise together, from the book of Bhatia. I could not found the  %minimax principle (in which form we are using here) as a one statement in %any book. If you know some better reference, please replace this Bhatia's %book from that one }
We now recall the following well known minimax principle.
\begin{prop} (Minimax principle ; \cite[Corollary III.1.2 \& Exercise III.1.3]{RB}) 
\label{p:maxmin}
Let $A$ be the self-adjoint operator on inner product space $(V, \langle \ \rangle)$. Let $\lambda_{min}$  be the minimum  eigenvalue of $A$. For $0 \neq x \in V$,
$$ \lambda_{min} \leq \frac{\langle Ax, x \rangle}{\langle x, x \rangle}.$$  
\end{prop}
\begin{proof}[Proof of Theorem \ref{t:subcomplex}.]
Let $0 \neq \phi \in C^{k}(X';\bR)$ be an eigenvector of $\Delta'_k = \Delta_k(X')$ with eigenvalue $\mu'_k = \mu_k(X')$. Using \eqref{e:equal1} and \eqref{e:lessequal1} for the first inequality below along with the definition of Laplacian and the minimax principle (Proposition \ref{p:maxmin}), we derive that
\begin{equation*} 
\begin{split}
\mu'_k ||\phi||^2 & = \langle \Delta'_k \phi, \phi \rangle = ||\delta'_k \phi||^2 + ||\delta_{k-1}^{' \ast} \phi||^2\\
& \geq ||\delta_k \phi||^2 + ||\delta_{k-1}^{\ast} \phi||^2 - (k+2)S_k(X,X') ||\phi||^2\\
& = \langle \Delta_k \phi, \phi \rangle - (k+2)S_k(X,X') ||\phi||^2\\
& \geq \mu_k ||\phi||^2 - (k+2)S_k(X,X') ||\phi||^2.
\end{split}
\end{equation*}
\end{proof}

\begin{proof}[Proof of Corollary \ref{cor:subcomplex}.]
 By applying induction on $k$ in Theorem \ref{t:abm1}, we derive that  $\mu_k(X) \geq (k+1)\mu_0(X) - kn$. Now, substituting the above bound in Theorem \ref{t:subcomplex} and using the fact that $\mu_0(X) = \lambda_2(G_X)$, we obtain
$$\mu_k(X') \geq (k+1)\lambda_2(G_X) -kn -(k+2)S_k(X,X').$$

Hence, if $\lambda_2(G_X) > \frac{kn}{k+1} + \frac{k+2}{k+1} S_k(X,X')$, then we have that $\mu_k(X') > 0$ and Proposition \ref{p:hodge} implies that $\widetilde{H}^k(X') = 0$.
 \end{proof}

For an $i$-simplex $\eta \in X$, let deg($\eta$) denote the number of $(i+1)$-simplices in $X$ which contain $\eta$. For $\phi \in C^k(X)$ and  a vertex $u \in V(X)$ define $\phi_u \in C^{k-1}(X;\bR)$ by 
$$
            \phi_u(\tau)= \begin{cases}
            \phi(u\tau)  & \text{if $u \in \lk_X(\tau) $},\\
\ 0 & \text{otherwise}.\\
                       \end{cases}
  $$

We now recall some results from \cite{ABM}, which were stated and proved for a clique complex but the same proofs also remain valid for any general simplicial complex. %For the sake of completeness, we give proofs of these results in the appendix (Section \ref{sec:App}).

\begin{claim} \label{claim3.1} \cite[Claim $3.1$]{ABM}
For $\phi \in C^k(X;\bR)$
\begin{equation} \label{e:ABM1}
||\delta_k \phi||^2 \ = \sum_{\sigma \in X(k)} \mbox{deg}(\sigma) \phi(\sigma)^2 - 2 \sum_{\eta \in X(k-1)} \sum_{vw\in \lk_X(\eta)} \phi(v\eta)\phi(w\eta).
\end{equation}
\end{claim}

\begin{claim} \label{claim3.3} \cite[Claim $3.3$]{ABM}
For $\phi \in C^k(X;\bR)$
\begin{equation} \label{e:ABM3}
\sum_{u \in V(X)} ||\delta_{k-2}^{\ast} \phi_u||^2 = k ||\delta_{k-1}^{\ast} \phi||^2.
\end{equation}

\end{claim}
\begin{claim} \label{claim3.2} \cite[page $7$, upto second equality in the proof of Claim $3.2$]{ABM}
\begin{equation}\label{e:ABM2}
\begin{split}
\sum_{u \in V(X)} ||\delta_{k-1} \phi_u||^2 & =  \sum_{\sigma \in X(k)} \Big(\sum_{\tau \in \sigma(k-1)} \mbox{deg}(\tau)\Big) \phi(\sigma)^2 \\
& \ \ \ \ - 2 \sum_{\eta \in X(k-2)} \sum_{vw \in \lk_X(\eta)} \sum_{u\in \lk_X(v\eta) \cap \lk_X(w\eta)} \phi(vu\eta)\phi(wu\eta).
\end{split}
\end{equation}
\end{claim}

 \begin{proof}[Proof of Theorem \ref{t:general}] 
 		Let $0 \neq \psi \in C^k(X;\bR)$ be an eigenvector of $\Delta_k$   with eigenvalue $\mu_k(X)$. By double counting 
 	\begin{align} \label{e:doublecounting}
 	\sum\limits_{v \in V(X)} ||\psi_v||^2 = (k+1)||\psi||^2.
 	\end{align}
 	
 	We first derive an expression for $\sum\limits_{u \in V(X)} \langle \Delta_{k-1} \psi_u, \psi_u \rangle$. We shall use \eqref{e:ABM2} in the second equality below. 
 	\begin{equation} \label{e:ABMequation1}
 	\begin{split}
 	\sum\limits_{u \in V(X)} \langle \Delta_{k-1} \psi_u, \psi_u \rangle  & = \sum\limits_{u \in V(X)}(||\delta_{k-1} \psi_u||^2 + ||\delta_{k-2}^{\ast}\psi_u||^2)\\
 	& =  \sum_{u \in V(X)}||\delta_{k-2}^{\ast}\psi_u||^2 + \sum_{\sigma \in X(k)} \Big(\sum_{\tau \in \sigma(k-1)} \mbox{deg}(\tau)\Big)\psi(\sigma)^2\\
 	& \ \ - 2 \sum_{\eta \in X(k-2)} \sum_{vw \in \lk_X(\eta)} \sum_{u \in \lk_X(v \eta) \cap \lk_X(w \eta)} \psi(vu \eta)\psi(wu \eta).
 	\end{split}
 	\end{equation}
Now, we relate $\sum\limits_{u \in V(X)} \langle \Delta_{k-1} \psi_u, \psi_u \rangle$ to $k\langle \Delta_k \psi, \psi \rangle$. In the following derivation, we shall use \eqref{e:ABM1} and \eqref{e:ABM3} for the second equality and the third equality will follow from \eqref{e:ABMequation1}.
 	\begin{equation*}
 	\begin{split}
 	k\langle \Delta_k \psi, \psi \rangle  & = k(||\delta_{k} \psi||^2 + ||\delta_{k-1}^{\ast}\psi||^2) \\
 	& = k\Big(\sum_{\sigma \in X(k)}\mbox{deg}(\sigma)\psi(\sigma)^2-2\sum_{\eta\in X(k-1)} \sum_{vw \in \lk_X(\eta)} \psi(v\eta)\psi(w\eta)\Big)\\
 	& ~~~~ + \sum_{u \in V(X)} ||\delta_{k-2}^{\ast}\psi_u||^2\\
 	& = k\sum_{\sigma \in X(k)}\mbox{deg}(\sigma)\psi(\sigma)^2 - 2k\sum_{\eta\in X(k-1)} \sum_{vw \in \lk_X(\eta)} \psi(v\eta)\psi(w\eta)\\
 	& ~~~~ + \sum_{u \in V(X)} \langle \Delta_{k-1}\psi_u, \psi_u\rangle  - \sum_{\sigma \in X(k)}\Big( \sum_{\tau \in \sigma(k-1)} \mbox{deg}(\tau)\Big) \psi(\sigma)^2 \\
 	& ~~~~ + 2 \sum_{\eta \in X(k-2)} \sum_{vw \in \lk_X(\eta)} \sum_{u \in \lk_X(v\eta) \cap \lk_X(w\eta)} \psi(vu\eta)\psi(wu\eta).
 	\end{split}
 	\end{equation*}
 	Thus, from the previous two derivations, we obtain that
 	$$ k \langle \Delta_k \psi, \psi \rangle  = \sum\limits_{u \in V(X)} \langle \Delta_{k-1}\psi_u, \psi_u \rangle + I_1-I_2-T,$$
 	where 
 	\begin{align}
 	T & := \sum_{\sigma \in X(k)}\Big( \sum_{\tau \in \sigma(k-1)} \mbox{deg}(\tau) - k \ \mbox{deg}(\sigma)\Big) \psi(\sigma)^2, \label{e:T}\\ 
 	I_1 & :=  2 \sum_{\eta \in X(k-2)} \sum_{vw \in \lk_X(\eta)} \sum_{u \in \lk_X(v\eta) \cap \lk_X(w\eta)} \psi(vu\eta)\psi(wu\eta) \label{e:I1}
 	\end{align}
 	and
 	\begin{align}
 	I_2 & := 2k\sum_{\eta\in X(k-1)} \sum_{vw \in \lk_X(\eta)} \psi(v\eta)\psi(w\eta). \label{e:I2}
 	\end{align}
We now use the bounds for $|I_1 -I_2|$ and $T$ given in  Claims \ref{i1i2} and \ref{T1} which we prove at the end of this section.  Combining Claims \ref{i1i2} and \ref{T1}, we have the following.
 \begin{align}
 k\langle\Delta_k \psi, \psi \rangle & \geq \sum_{v \in V(X)} \langle\Delta_{k-1}\psi_v, \psi_v\rangle -(|V(X)|+ \sum\limits_{j=2}^{k+1} D_k(X,j) (k(k+1)+j)) ||\psi||^2. \label{e:Deltak}
 \end{align}
  
From  \eqref{e:Deltak}, Minimax principle (Proposition \ref{p:maxmin}) and \eqref{e:doublecounting} we have 
\begin{align*}
k \mu_k(X) ||\psi||^2 & = k\langle\Delta_k \psi, \psi\rangle  \geq \sum_{v \in V(X)} \langle \Delta_{k-1}\psi_v, \psi_v \rangle-(n+ \sum\limits_{j=2}^{k+1}(k(k+1)+j) D_k(X,j)) ||\psi||^2 \\
& \geq \mu_{k-1}(X) \sum_{v \in V(X)} ||\psi_v||^2 -  (n+ \sum\limits_{j=2}^{k+1}(k(k+1)+j) D_k(X,j)) ||\psi||^2 \\
& = ((k+1) \mu_{k-1}(X) - n -  \sum\limits_{j=2}^{k+1} (k(k+1)+j)D_k(X,j)) ||\psi||^2.
\end{align*}
\end{proof}
 
\begin{proof}[Proof of Corollary \ref{cor:general}] Since the $k$-skeleton of $X$ is same as that of the clique complex of $G_X$, we have from Remark \ref{r:Dk} that $D_i(X,i) = 0$ for all $2 \leq i \leq k$. Hence, Theorem \ref{t:general} implies that $\mu_k(X) \geq (k+1) \mu_0(X) - kn - (k+1)^2 D_k(X,k+1)$. Therefore, if $\mu_0(X) = \lambda_2(G_X) > \frac{kn}{k+1} + (k+ 1) D_k(X,k+1)$, then $\mu_k(X) > 0$ and result follows from Proposition \ref{p:hodge}.
\end{proof}
 
In the rest of the paper, we use $ \1[ \  \cdot \ ]$ to denote an indicator function. Now, we shall give proofs of Claims \ref{i1i2} and \ref{T1}. For a $k$-simplex $\sigma$  and  $v_1, \ldots, v_l \in \sigma$, $\hat{\sigma}_{v_1 \ldots v_l} := \sigma \setminus \{v_1, \ldots, v_l\}$ is a ($k-l$)-simplex. Also, for simplices $\sigma$ and $\eta$, by $\sigma \triangle \eta$ we refer to their symmetric difference as sets of vertices.
Recall that the definitions of $D_i(X, j), i \geq 1, 1 \leq j \leq i+1$ are given in \eqref{e:Dkj} and  the definitions of $T, I_1$ and $I_2$ can be found in \eqref{e:T}, \eqref{e:I1} and \eqref{e:I2} respectively.
 
 %\blue{Here,  $T, I_1$ and $I_2$ dependes on $\phi$, so  should we change these notions to $T^{\phi}, I_1^{\phi}$ and $I_2^\phi$ respectively ?}
 
 \begin{claim} \label{i1i2}
 	\begin{equation} \label{I2I1}
 	|I_1-I_2| \leq k(k+1) \sum\limits_{j=2}^{k+1} D_k(X,j) ||\psi||^2.
 	\end{equation}
 	
 \end{claim}
 \begin{proof} In this  proof, we use the convention that $\psi(\tau) = 0$, whenever $\tau \notin  X(k)$. 
 	Observe that the expression for $I_2$ given in \eqref{e:I2} can be rewritten as, 
 	\begin{align*}
 	I_2 = 2 \sum_{\eta \in X(k-2)}  \sum_{\substack{\{v,w\} \\ vw \in \lk_X(\eta)}} \sum_{\substack{u \in \lk_X(v\eta) \cap \lk_X(w\eta) \\ u \in \lk_X(vw\eta)}} \psi(vu\eta)\psi(wu\eta).
 	\end{align*}
 	By recalling the definition of $I_1$  from \eqref{e:I1}, we obtain,
% 	\begin{align*}
% 	I_1 & =   2 \sum_{\eta \in X(k-2)}  \sum_{\substack{\{v,w\} \\ vw \in lk_X(\eta)}} \sum_{\substack{u \in lk_X(v\eta) \cap lk_X(w\eta) \\ u \notin lk_X(vw\eta)}} \psi(vu\eta)\psi(wu\eta) \\
% 	& ~~~~ +  2 \sum_{\eta \in X(k-2)}  \sum_{\substack{\{v,w\} \\ vw \in lk_X(\eta)}} \sum_{\substack{u \in lk_X(v\eta) \cap lk_X(w\eta) \\ u \in lk_X(vw\eta)}} \psi(vu\eta)\psi(wu\eta)\\
% 	& =  2 \sum_{\eta \in X(k-2)}  \sum_{\substack{\{v,w\} \\ vw \in lk_X(\eta)}} \sum_{\substack{u \in lk_X(v\eta) \cap lk_X(w\eta) \\ u \notin lk_X(vw\eta)}} \psi(vu\eta)\psi(wu\eta) + I_2.
% 	\end{align*} 
% 	Thus,
 	\begin{align*}
 	I_1-I_2 & = 2 \sum_{\eta \in X(k-2)}  \sum_{\substack{\{v,w\} \\ vw \in \lk_X(\eta)}} \sum_{\substack{u \in \lk_X(v\eta) \cap \lk_X(w\eta) \\ u \notin \lk_X(vw\eta)}} \psi(vu\eta)\psi(wu\eta)\\
 	&= 2 \sum_{\sigma \in X(k)} \sum_{\{v, w\} \subseteq \sigma}  \sum_{\substack{u \in \lk_X(\hat{\sigma}_w) \cap \lk_X(\hat{\sigma}_v) \\ u \notin \lk_X(\sigma)}} \psi(uv\hat{\sigma}_{vw})\psi(uw\hat{\sigma}_{vw})\\
 	& = 2 \sum_{\sigma \in X(k)} \sum_{u \notin \lk_X(\sigma)} \sum_{\{v, w\} \subseteq \sigma} \1[u \in \lk_X(\hat{\sigma}_v) \cap \lk_X(\hat{\sigma}_w)]\psi(uv\hat{\sigma}_{vw})\psi(uw\hat{\sigma}_{vw}) \\
 	 &= 2 \sum_{\sigma \in X(k)} \sum\limits_{j = 2}^{k+1} \sum_{u \notin \lk_X(\sigma)} \1[u \in \bigcap\limits_{i=1}^{j} \lk_X(\hat{\sigma}_{v_i})\ \mbox{for exactly} \ j \ \mbox{vertices} \ v_1, \ldots, v_j \in \sigma]\\
 	& ~~~~~\sum_{\{v, w\} \subseteq \sigma} \1[u \in \lk_X(\hat{\sigma}_v) \cap \lk_X(\hat{\sigma}_w)]\psi(uv\hat{\sigma}_{vw})\psi(uw\hat{\sigma}_{vw})\\
% 	&= 2 \sum_{\{v_1, \ldots, v_{k+2}\} \notin X(k+1)}\sum_{i \in [k+2]} \1[\gamma^i = \{v_1, \ldots, \hat{v_i}, \ldots, v_{k+2}\} \in X(k)]\\
% 	& ~~~~~\sum\limits_{j=2}^{k+1} \1[v_i \in \bigcap\limits_{l=1}^{j} \lk_X(\hat{\gamma^i}_{v_{i_l}}) \ \mbox{for exactly} \ j \ %  \mbox{vertices} \ v_{i_1}, \ldots, v_{i_j} \in \gamma^i] \\
% 	& ~~~~~ \sum_{i \notin \{p, q\} \subseteq[k+2]} \1[v_i \in \lk_X(\hat{\gamma^i}_{v_p}) \cap \lk_X(\hat{\gamma^i}_{v_q}) ]\psi(v_iv_p %\hat{\gamma^i}_{v_pv_q})\psi(v_iv_q\hat{\gamma^i}_{v_pv_q}) \\
 & =  2 \sum_{\{v_1, \ldots, v_{k+2}\} \notin X(k+1)} \sum_{i \in [k+2]} \1[\gamma^i = \{v_1, \ldots, \hat{v_i}, \ldots, v_{k+2}\} \in X(k)] \\
 	& ~~~~~ \sum\limits_{j=2}^{k+1} \1[v_i \in \bigcap\limits_{l=1}^{j} \lk_X(\hat{\gamma^i}_{v_{i_l}}) \ \mbox{for exactly} \ j \ \mbox{vertices} \ v_{i_1}, \ldots, v_{i_j} \in \gamma^i] \\ &~~~~~  \sum_{ \{p, q\} \subset [k+2] \setminus\{i\}} \psi(v_iv_p \hat{\gamma^i}_{v_pv_q})\psi(v_iv_q\hat{\gamma^i}_{v_pv_q}) \notag.
 	\end{align*}
 	%\end{equation*}
 	 Hence, using \eqref{e:inequality} we obtain 
 	%\vspace{-0.2 cm}
 	\begin{align}
 	|I_1-I_2| & \leq 2\cdot\frac{k}{2}\sum_{\{v_1, \ldots, v_{k+2}\} \notin X(k+1)} \sum_{i \in [k+2]} \1[\gamma^i = \{v_1, \ldots, \hat{v_i}, \ldots, v_{k+2}\} \in X(k)]\nonumber\\
	& ~~~~~ \sum\limits_{j=2}^{k+1} \1[v_i \in \bigcap\limits_{l=1}^{j} 
 \lk_X(\hat{\gamma^i}_{v_{i_l}}) \ \mbox{for exactly} \ j \ \mbox{vertices} \      v_{i_1}, \ldots, v_{i_j} \in \gamma^i] \nonumber\\ 
	&~~~~~~   \sum_{p \in [k+2]\setminus\{i\}} \psi(v_i\hat{\gamma^i}_{v_{p}})^2\nonumber\\
	%& = k \sum_{\sigma \in X(k)} \psi(\sigma)^2 \sum\limits_{j=2}^{k+1} \\
	%&~~~~~ \sum\limits_{w \notin \lk_X(\sigma)} \1[w \in \bigcap\limits_{l=1}^{j} %%\lk_X(\hat\sigma_{v_{i_l}}) \ \mbox{for exactly} \ j \ \mbox{vertices} \ v_{i_1}, %%\ldots, v_{i_j} \in \sigma] \nonumber \\
	& = k\sum_{\sigma \in X(k)} \sum\limits_{w \notin \lk_X(\sigma)} \nonumber \\
	&~~~~~ \sum\limits_{j=2}^{k+1} \1[w \in \bigcap\limits_{l=1}^{j} \lk_X(\hat\sigma_{v_{i_l}}) \ \mbox{for exactly} \ j \ \mbox{vertices} \ v_{i_1}, \ldots, v_{i_j} \in \sigma] \sum\limits_{v \in \sigma} \psi(w\hat\sigma_{v})^2 \nonumber \\
	&= k \sum\limits_{j=2}^{k+1} \sum_{\eta \in X(k)} \psi(\eta)^2 \sum_{\sigma \in X(k)} \sum_{v,w} \1[w\hat\sigma_v \triangle \eta = \emptyset, w \notin \lk_X(\sigma),v \in \sigma] \nonumber\\
	&~~~~~ \times \1[w \in \bigcap\limits_{l=1}^{j} \lk_X(\hat\sigma_{v_{i_l}}) \ \mbox{for exactly} \ j \ \mbox{vertices} \ v_{i_1}, \ldots, v_{i_j} \in \sigma]  \nonumber\\
	& = k \sum\limits_{j=2}^{k+1} \sum_{\eta \in X(k)} \psi(\eta)^2 \sum_{\sigma \in X(k)} \sum_{v,w} \1[v\hat\eta_w \triangle  \sigma = \emptyset,v \notin \lk_X(\eta),w \in \eta] \nonumber \\
	&~~~~~ \times \1[v \in \bigcap\limits_{l=1}^{j} \lk_X(\hat\eta_{v_{i_l}}) \ \mbox{for exactly} \ j \ \mbox{vertices} \ v_{i_1}, \ldots, v_{i_j} \in \eta], \nonumber
	\end{align} 
where the last equality is justified as follows : For $\eta , \sigma \in X(k)$, we have that $w \notin \lk_{X}(\sigma), v \in \sigma, w \hat{\sigma}_v  \triangle \eta = \emptyset$ is equivalent to $v \hat{\eta}_{w} \triangle \sigma = \emptyset, w \in \eta, v \notin \lk_X(\eta)$. Now, let $w \in \bigcap\limits_{l=1}^{j} \lk_X(\hat\sigma_{v_{i_l}})$  for exactly $  j$ vertices $ v_{i_1}, \ldots, v_{i_j} \in \sigma$. Clearly, $v \in \{v_{i_1}, \ldots, v_{i_l}\}$ and without loss of generality, assume $v = v_{i_1}$. Thus, $v \in \lk_{X}(\hat{\eta}_{w}) \cap  \bigcap\limits_{j=2}^{l} \lk_X(\hat\eta_{v_{i_j}} )$. Suppose there exists $u \in \eta, u \notin \{v_{i_1}, \ldots, v_{i_j}, w\}$ such that $v \in \lk_{X}(\hat{\eta}_{u})$. Since, $(\hat{\eta}_u \cup \{v\}) \triangle (\hat{\sigma}_{u} \cup \{w\}) = \emptyset$, we conclude that $w \in \lk_{X}(\hat{\sigma}_{u})$, which contradicts the fact that $w  \in \bigcap\limits_{l=1}^{j} \lk_X(\hat\eta_{v_{i_l}})$ for exactly $j $ vertices $ v_{i_1}, \ldots, v_{i_j} \in \sigma$. 
	
	Thus we have, 
	\begin{align}
|I_1-I_2| & \leq 	  k \sum\limits_{j=2}^{k+1} \sum_{\eta \in X(k)} \psi(\eta)^2 \sum_{v,w} \sum_{\sigma \in X(k)} \1[v\hat\eta_w \triangle \sigma = \emptyset]\1[v \notin \lk_X(\eta),w \in \eta] \nonumber\\
	&~~~~~ \times \1[v \in \bigcap\limits_{l=1}^{j} \lk_X(\hat\eta_{v_{i_l}}) \ \mbox{for exactly} \ j \ \mbox{vertices} \ v_{i_1}, \ldots, v_{i_j} \in \eta] \nonumber\\
	&	 = k \sum\limits_{j=2}^{k+1} \sum_{\eta \in X(k)} \psi(\eta)^2 \sum_{w}\1[w \in \eta] \sum_v \1[v \notin \lk_X(\eta)] \nonumber\\
	&~~~~~ \times \1[v \in \bigcap\limits_{l=1}^{j} \lk_X(\hat\eta_{v_{i_l}}) \ \mbox{for exactly} \ j \ \mbox{vertices} \ v_{i_1}, \ldots, v_{i_j} \in \eta]\nonumber \\
	&	 \leq  k \sum\limits_{j=2}^{k+1} D_k(X,j) \sum_{\eta \in X(k)} \psi(\eta)^2 \sum_{w}\1[w \in \eta]  \nonumber\\	
	& = k(k+1)\sum\limits_{j=2}^{k+1} D_k(X,j) \sum_{\eta\in X(k)} \psi(\eta)^2   \nonumber \\
 	& = k(k+1)\sum\limits_{j=2}^{k+1} D_k(X,j)||\psi||^2 .\nonumber
 	\end{align}
  \end{proof}

\begin{claim} \label{T1}
 	\begin{align} \label{T}
 	T & \leq (|V(X)|+  \sum\limits_{j=2}^{k+1} j D_k(X,j)) ||\psi||^2.
 	\end{align}
 \end{claim}  
 
 \begin{proof}
 	\begin{align*}
 	T & = \sum_{\sigma \in X(k)} \Big(\sum_{\tau \in \sigma(k-1)} \mbox{deg}(\tau) - k \  \mbox{deg}(\sigma) \Big) \psi(\sigma)^2 \\
 	& = \sum_{\sigma \in X(k)} \Big(\sum_{v \in \sigma} \sum_{u \in \lk_X(\hat{\sigma}_v)} \psi(\sigma)^2 - k \sum_{u \in \lk_X(\sigma)} \psi(\sigma)^2 \Big)\\ 
 	& = \sum_{\sigma \in X(k)} \sum_{u \in \lk_X(\sigma)} \psi(\sigma)^2 + \sum_{\sigma \in X(k)} \sum_{v \in \sigma}\sum_{\substack{u \in \lk_X(\hat{\sigma}_v) \\ u \notin \lk_X(\sigma)}} \psi(\sigma)^2 \\ 
 		& = \sum_{\sigma \in X(k)} \sum_{u \in \lk_{X}(\sigma)} \psi(\sigma)^2 + \sum_{\sigma \in X(k)} \Big( \sum_{v\in \sigma} |\{u | u \notin \lk_X(\sigma) \ \text{and} \ u \in \lk_X(\hat{\sigma}_v)\}| \Big) \psi(\sigma)^2\\
 	& = T_1 + T_2
 	\end{align*}
 	where, 
 	%\vspace{- .2 cm}
 	\begin{align}
 	T_1	& := \sum_{\sigma \in X(k)} \Big(\sum_{u \in \lk_{X}(\sigma)} \psi(\sigma)^2 + |\{ u| u \notin \lk_X(\sigma), u \in \lk_X(\hat{\sigma}_v) \ \mbox{for exactly one} \  v \in \sigma\}| \psi(\sigma)^2\Big) \nonumber\\
 	& \leq |V(X)| \sum_{\sigma \in X(k)} \psi(\sigma)^2 = |V(X)| ||\psi||^2 \nonumber
 	\end{align}
 	and
 	\begin{align}
 	T_2 & := \sum_{\sigma \in X(k)} \sum\limits_{j=2}^{k+1} j | \{u |  u \notin \lk_X(\sigma), u \in \bigcap\limits_{i=1}^{j} \lk_X(\hat{\sigma}_{v_i})\ \mbox{for exactly} \ j \ \mbox{vertices} \ v_1, \ldots, v_j \in \sigma \}| \psi(\sigma)^2 \nonumber\\
 	& \leq \sum_{\sigma \in X(k)} \sum\limits_{j=2}^{k+1}  jD_k(X,j)\psi(\sigma)^2 \nonumber\\
 	& \leq  \sum\limits_{j = 2}^{k+1}j D_k(X,j)  ||\psi||^2 \nonumber.
\end{align}
\end{proof}
 
 %\blue{we need to change the name of this subsection as this name is already of section \ref{sub:randomneighborhood}}
 
 \subsection{Proofs of the results in section \ref{subsec:randomneighborhood}}
 \label{sec:random}
\begin{proof}[Proof of Theorem \ref{t:main}.] 
Let $p = \Big(\frac{(k+1)\log n  +c_n}{n}\Big)^{\frac{1}{k+2}}$, where  $c_n \to \infty$. The two main steps of the proof are (i) $\mathcal{N}(G(n, p))$ has full $k$-skeleton and (ii) $n^{-1}\mathbb{E}(D_k(\N(G(n,p)),k+1)) \to 0$. 
  
From (ii) and Markov's inequality, we have that for all $\epsilon > 0$, $\bP(D_k(\N(G(n,p)),k+1) \geq \epsilon \frac{ n}{(k+1)^2}) \to 0$ and hence w.h.p. $D_k(\N(G(n,p)),k+1) < \frac{ n}{(k+1)^2}$. Let $G$ be the $1$-skeleton of $\mathcal{N}(G(n, p))$. Since $G$ is a complete graph w.h.p. (because of (i)), the spectral gap $\lambda_2(G) = n$ and therefore Corollary \ref{cor:general} and (i) imply that $\widetilde{H}^k(\mathcal{N}(G(n, p))) = 0$. This completes the proof provided we establish (i) and (ii).

First, we show (i). The expected number of  $(k+1)$-tuples of  vertices in $G(n, p)$  with no neighbor is
 \begin{align*} 
 {n \choose k+1}(1-p^{k+1})^{n-k-1} & \leq {n \choose k+1} e^{-(n-k-1)p^{k+1}} \\
 & = {n \choose k+1} e^{-n\big(\frac{(k+1)\log n + c_n}{n}\big)^{\frac{k+1}{k+2}}} e^{(k+1)\big(\frac{(k+1)\log n + c_n}{n}\big)^{\frac{k+1}{k+2}}}\\
 &={n \choose k+1}e^{-n^{\frac{1}{k+2}} ((k+1)\log n + c_n)^{\frac{k+1}{k+2}}}e^{(k+1)\big(\frac{(k+1)\log n + c_n}{n}\big)^{\frac{k+1}{k+2}}}\\
 & = o(1).
 \end{align*}
 Hence, w.h.p. $\mathcal{N}(G(n, p))$ has full $k$-skeleton. In particular, w.h.p. $k$-skeleton of $\N(G(n, p))$ is a clique complex. It is well known that, if a complex has full $k$-skeleton then it has trivial cohomology in all dimensions less than $k$. Therefore $\widetilde{H}^i(\mathcal{N}(G(n, p))) = 0$ for $i < k$. 
  
Now, we shall establish (ii). Let $B_k$ be the number of subcomplexes of $\N(G(n, p))$, which are isomorphic to the simplicial boundary of a $(k+1)$-simplex. Since the $k$-skeleton of $\N(G(n,p))$ is a clique complex, we observe from \eqref{def:Dk} that  $D_k(\N(G(n, p)),k+1) \leq B_k$. Thus it suffices to show that $n^{-1}\bE(B_k) \to 0$ to prove (ii).  
  
The rest of the proof is to compute $\bE(B_k)$ and show the above. For $0 \leq i \leq  { k+2 \choose 2}$, let $C_i$ denote the number of graphs on $k+2$ vertices with $i$ edges. For any $\{v_1, \ldots, v_{k+2}\} \subset [n]$, the probability that the induced subcomplex of $\N(G(n,p))$ on 
 $\{v_1, \ldots, v_{k+2}\}$ is isomorphic to the simplicial boundary of a $(k+1)$-simplex, is bounded above by  $(1-p^{k+2})^{n-k-2}$. Therefore
 \begin{align*} 
 %\label{expectation}
 \frac{\mathbb{E} B_k}{n}  & \leq \frac{1}{n} {n \choose k+2} \sum\limits_{i=0}^{{k+2 \choose 2}}C_i p^i(1-p)^{{k+2 \choose 2} - i} (1-p^{k+2})^{n-k-2} \\
& \leq   \frac{1}{n}{n \choose k+2} \sum\limits_{i=0}^{{k+2 \choose 2}} C_i p^i(1-p)^{{k+2 \choose 2} - i} e^{-(n-k-2)p^{k+2}}\\
 & =  \frac{1}{n}{n \choose k+2} \sum\limits_{i=0}^{{k+2 \choose 2}} C_i p^i(1-p)^{{k+2 \choose 2} - i} e^{-(n-k-2)\frac{(k+1)\log n + c_n}{n}}\\
& =   {n \choose k+2} n^{-(k+2)} e^{-c_n}  e^{(k+2)\frac{(k+1)\log n + c_n}{n}} \sum\limits_{i=0}^{{k+2 \choose 2}} C_i  (1-p)^{{k+2 \choose 2} - i} p^i.
 \end{align*}
Now using the fact that $c_n \to \infty$ and $p \to 0$, we derive that $n^{-1}\bE(B_k) \to 0$ and thereby completing the proof of (ii) as well as that of the theorem. 
\end{proof}
Let $U = \{u_1, \ldots, u_r\}, V = \{v_1, \ldots, v_r\}$ be subsets of $[n]$ such that $U \cap V = \emptyset.$ The graph $X_{U,V}$ is defined as the graph with vertex set $U \cup V$ and edges $u_i \sim u_j$ and $u_i \sim v_j$ for all $ 1 \leq i \neq j \leq r$. To prove Proposition \ref{p:extension}, we need the following result relating to cohomology to graphs isomorphic to $X_{U,V}$ graphs. For two sequences of real numbers $a_n, b_n, n \geq 1$, we shall use the notation $a_n = o(b_n)$ to denote that $a_n / b_n \to 0$ as $n \to \infty$.
 
 \begin{prop} \label{kahleprop}\cite[Theorem $2.7$]{kahle}
 If $H$ is any graph containing a maximal clique of order $r$ that cannot be extended to an $X_{U,V}$ subgraph for some $U,V$, then $\N(H)$ retracts onto a sphere $\mathbb{S}^{r-2}$.
 \end{prop}
\begin{proof}[Proof of Proposition \ref{p:extension}.]
 Let $p  = n^{\alpha}$, where $ \frac{-2}{r-1} < \alpha  < \frac{-1}{r-1}, r \geq 2$. We shall set $G_n = G(n,p)$. Let 
\begin{align*}
\Lambda_r := & |\{ A \subset [n] \ | \ |A| = r, G_n[A] \  \text{is a maximal clique and} \  G_n[A] \not\subseteq  X_{A,A'} \ \text{for all}  \ A' \  \text{disjoint} \\
& \text{and} \ |A'| = r \}|.
 \end{align*}
 For a $A \subset [n]$ with $|A| = r$, let
 \begin{align*}
  I_A & := \1[\mbox{$G_n[A]$ is a $r$-clique}], \, \, J_A := \prod_{A' : A' \supsetneq A} \1[\mbox{$G_n[A']$ is not a clique}], \\
   K_A  &:= \prod_{ \substack{ A' : |A'| = r  \\ A' \cap A = \emptyset}}\1[[G_n[A] \not\subseteq X_{A,A'} ].  
  \end{align*} 

Then,  $\Lambda_r = \sum\limits_{A \subset [n], |A| = r} I_A J_A K_A$. By Proposition \ref{kahleprop}, the proof is complete provided we show that $\Lambda_r \geq 1$ w.h.p.. We shall use the following second moment bound to show the latter :
$$ \mathbb{P}(\Lambda_r \geq 1) \geq \frac{(\mathbb{E} \Lambda_r)^2} {\mathbb{E}\Lambda_r^2}.$$

To use the second moment bound, we first derive a lower bound for $\bE(\Lambda_r)$. Fix $A = [r]$ in the below derivation. 
%
%\begin{align}
%\bE(\Lambda_r) &= {n \choose r} \bE(I_AJ_AK_A) = {n \choose r} \bE\left(I_A (J_A - J_A\1[\cup_{\substack{A' : |A'| = r \\ A' \cap A =\emptyset}} \{G_n[A] \subset X_{A,A'} \} ]) \right) \no \\
%& \geq {n \choose r} \bE\left(I_A \left( J_A - \right. \right \no \\
%& \left. \left.\sum_{\substack{A' = \{u_1, \ldots, u_r\} \\ A' \cap A =\emptyset}} \1[G_n[A] \subset X_{A,A'} ] \1[i \nsim u_i \ \forall i] \prod_{v \notin A' \cup A}\1[ i \nsim v \ \mbox{for some} \ i \in A] \right) \right)  \no \\
%& \geq  {n \choose r} \bE\left( \1[i \sim j \ \forall \ 1 \leq i  \neq j \leq r] \bE\left( \prod_{v \notin A}\1[ i \nsim v \ \mbox{for some $i \in A$}] - \right. \right. \no \\
%&  \left. \left. \sum_{\substack{A' = \{u_1, \ldots, u_r\} \\ A' \cap A =\emptyset}} \1[u_i \sim j \ \forall \ 1 \leq i  \neq j \leq r]\1[i \nsim u_i \ \forall i] \prod_{v \notin A' \cup A}\1[ i \nsim v \ \mbox{for some $i \in A$}] \right) \right)  \no \\
%& = {n \choose r} p^{{r \choose 2}} ((1-p^r)^{n-r} - {n-r \choose r} p^{r(r-1)}(1-p)^r(1-p^r)^{n-2r} ) \no \\
%& \geq {n \choose r} p^{{r \choose 2}}(1-p^r)^{n-r}(1 - n^{r} p^{r(r-1)}(1-p)^r(1-p^r)^{-r}), \no
%\end{align}
\begin{align}
\bE(\Lambda_r) &= {n \choose r} \bE(I_AJ_AK_A) = {n \choose r} \bE\left(I_A (J_A - J_A\1[\cup_{\substack{A' : |A'| = r \\ A' \cap A =\emptyset}} \{G_n[A] \subset X_{A,A'} \} ]) \right) \no \\
& \geq {n \choose r} \bE \Bigg( I_A \Bigg( J_A -  \no \\
& ~~~~~~~~~~ \sum_{\substack{A' = \{u_1, \ldots, u_r\} \\ A' \cap A =\emptyset}} \1[G_n[A] \subset X_{A,A'} ] \1[i \nsim u_i \ \forall i] \prod_{v \notin A' \cup A}\1[ i \nsim v \ \mbox{for some} \ i \in A] \Bigg) \Bigg)  \no \\
& \geq  {n \choose r} \bE\left( \1[i \sim j \ \forall \ 1 \leq i  \neq j \leq r] \bE\left( \prod_{v \notin A}\1[ i \nsim v \ \mbox{for some $i \in A$}] - \right. \right. \no \\
&  \left. \left. \sum_{\substack{A' = \{u_1, \ldots, u_r\} \\ A' \cap A =\emptyset}} \1[u_i \sim j \ \forall \ 1 \leq i  \neq j \leq r]\1[i \nsim u_i \ \forall i] \prod_{v \notin A' \cup A}\1[ i \nsim v \ \mbox{for some $i \in A$}] \right) \right)  \no \\
& = {n \choose r} p^{{r \choose 2}} ((1-p^r)^{n-r} - {n-r \choose r} p^{r(r-1)}(1-p)^r(1-p^r)^{n-2r} ) \no \\
& \geq {n \choose r} p^{{r \choose 2}}(1-p^r)^{n-r}(1 - n^{r} p^{r(r-1)}(1-p)^r(1-p^r)^{-r}), \no
\end{align}
where in the equality in the penultimate line we have used the independence between the corresponding indicator random variables as they depend on disjoint sets of edges. Since $p = n^{\alpha}$ for $\alpha < \frac{-1}{r-1}$, we have that $n^{r} p^{r(r-1)}(1-p)^r(1-p^r)^{-r} \to 0$ and hence for large enough $n$, we derive that 
\begin{equation}
\label{eq:ELambda}
\bE(\Lambda_r) \geq  {n \choose r} p^{{r \choose 2}}(1-p^r)^{n-r} (1 - o(1)).
\end{equation}
%
%Now, we shall lower bound the last summand as follows. Let us fix $A' = \{u_1, \ldots, u_r\}$ such %that $A' \cap [r] =\emptyset$.
%\begin{align}
%\mathbb{E}K_A & = 1 -\mathbb{P} ([G_n[A] \subseteq X_{A,A'})  \notag \\
% & \geq 1- \sum_{ \substack{A' = \{u_1, \ldots, u_{|A|}\} \\ A \cap A' = \emptyset}} \mathbb{P}( \{u_i %\sim v_j, u_i \not\sim v_i \ \forall \ 1 \leq i  \neq j \leq |A|\})  \notag \\  
%& \approx 1 - n^{|A|} p^{|A|(|A|-1)} (1-p)^{|A|}\label{kaexpectation}.
%\end{align} 
% 
% Using \eqref{kaexpectation}, we conclude that 
%
%\begin{align} \label{eq:ELambda}
  %\mathbb{E} \Lambda_r \geq {n \choose r} p^{{r \choose 2}}(1 - n^{r} p^{r(r-1)} (1-p)^{r}).
%\end{align}
%  
Now, we proceed to derive upper bounds for the second moment. 
$$ \Lambda_r^2 =\sum\limits_{i = 0}^{r} \sum\limits_{\substack{ |A_1| = |A_2| = r \\ |A_1 \cap A_2| = i } } I_{A_1}J_{A_1} K_{A_1} I_{A_2} J_{A_2}K_{A_2}  \leq \sum\limits_{i=0}^r Y_i,$$
where $Y_i := \sum\limits_{\substack{|A_1| = |A_2| = r \\ |A_1 \cap A_2| = i} } I_{A_1} I_{A_2}$ for $0 \leq i \leq r$.
 Hence, using \eqref{eq:ELambda}, we derive that for large enough $n$
   \begin{align}
  \frac{\mathbb{E}\Lambda_r^2} {(\mathbb{E} \Lambda_r)^2} &\leq \frac{1} {(\mathbb{E} \Lambda_r)^2}  \sum\limits_{i=0}^r \mathbb{E} Y_i \nonumber \\
  & \leq \frac{1}{({n \choose r} p^{{r \choose 2}}(1-p^r)^{n-r} (1 - o(1)))^2} \sum\limits_{i=0}^r{n \choose 2r-i} p^{r(r-1) - \frac{i(i-1)}{2}}  \nonumber \\
  & = \frac{{n \choose 2r} }{{n \choose r}^2(1-p^r)^{2(n-r)}(1 - o(1))^2}  +  \sum\limits_{i=1}^r{n \choose 2r-i} \frac{ p^{r(r-1) - \frac{i(i-1)}{2}}  }  {{n \choose r}^2 p^{r(r-1)}(1-p^r)^{2(n-r)}(1 - o(1))^2} \nonumber \\
  & = \frac{{n \choose 2r} }{{n \choose r}^2(1-p^r)^{2(n-r)}(1 - o(1))^2}  +  \sum\limits_{i=1}^r \frac{  {n \choose 2r-i}  }  {{n \choose r}^2 n^{\alpha \frac{i(i-1)}{2}}(1-p^r)^{2(n-r)}(1 - o(1))^2} \nonumber \\
  & \leq  \frac{{n \choose 2r}}{{n \choose r}^2(1-p^r)^{2(n-r)}(1 -o(1))^2}  + C \sum\limits_{i=1}^r \frac{n^{2r}} {{n \choose r}^2 n^{ i + \alpha \frac{i(i-1)}{2}}(1-p^r)^{2(n-r)}(1 -o(1))^2} \nonumber \\
  & = 1 + o(1) \label{lambda0},
   \end{align}
where $C$ is a constant and \eqref{lambda0} follows as $1+\alpha(r-1) <  0$, $i + \alpha \frac{i(i-1)}{2} > 0$ (because $\alpha > \frac{-2}{r-1}$) and further $(1-p^r)^{2(n-r)} = e^{-2np^r}(1 + o(1)) = 1 + o(1)$ for large $n$. 
%and therefore  $\frac{{n \choose 2r-i}}{{n \choose r}^2 n^{\alpha \frac{i(i-1)}{2}}} \approx \frac{1}{n^{i + \alpha \frac{i(i-1)}{2}}} \to 0$.
   
 From \eqref{lambda0}, we conclude that $\liminf_{n \to  \infty}\frac{(\mathbb{E} \Lambda_r)^2} {\mathbb{E}\Lambda_r^2} = 1$. Therefore by the second moment bound, we derive that $\lim_{n \to \infty}\mathbb{P}(\Lambda_r \geq 1) = 1$ as required.
\end{proof}
 
\section*{Acknowledgements}
   
   SS was financially supported by the Indian Statistical Institute Bangalore, India, where this work was done and the DST INSPIRE Faculty award of Dr. D. Y.. D.Y. was supported by the DST INSPIRE Faculty award and CPDA from the Indian Statistical Institute. The authors thank the two anonymous referees for many comments leading to an improved exposition.
   
 \bibliographystyle{plain}

\begin{thebibliography}{99}
\begin{small}
%\bibitem{BK}  Eric  Babson and Dmitry  N. Kozlov, \newblock\emph{Complexes of graph homomorphisms}. \newblock{Israel J. Math.} Vol 152, 2006, 285--312.

\bibitem{ABM} R. Aharoni, E. Berger and R. Meshulam, \newblock\emph{Eigenvalues and homology of flag complexes and vector representations of graphs}. 
\newblock{Geom. Funct. Anal}. 15, no. 3, 555--566, 2005.

\bibitem{BK}  E.  Babson and D. Kozlov, \newblock\emph{Complexes of graph homomorphisms}. \newblock{Israel J. Math.}, 152, no. 1, 285--312, 2006.

\bibitem{BK1} E. Babson and D. Kozlov, \newblock\emph{Proof of the Lov\'asz conjecture}. \newblock{Ann. Math. (2)}, 165, 965--1007, 2007.

\bibitem{bapat} R. B. Bapat, \newblock{Graphs and matrices}, 27, New York: Springer, 2010.

\bibitem{RB} R. Bhatia, \newblock{ Matrix analysis}. Vol. 169. Springer Science \& Business Media, 2013.


 \bibitem{Herbert} H. Edelsbrunner  and J. Harer, \newblock{Computational topology: an introduction},\newblock{American Mathematical Society}, 2010.

\bibitem{BS} W. Ballamann and J. {\' S}wi\c{a}tkowski, \newblock\emph{On $L^2$-cohomology and property (T) for automorphism groups of polyhedral cell complexes}. \newblock{Geom. Funct. Anal.}, 7, no. 4, 615--645, 1997.

\bibitem{Frieze16} A. Frieze and M. Karo\'{n}ski. Introduction to random graphs. Cambridge University Press, 2016.

\bibitem{G} H. Garland, \newblock\emph{p-adic curvature and the cohomology of discrete subgroups of p-adic groups}. \newblock{Ann. Math.} 97, 375--423, 1973.

%\bibitem{h} A. Hatcher, \newblock\emph{Algebraic topology}. Cambridge University Press, 2002.

\bibitem{HK} M. Hino and S. Kanazawa, \newblock\emph{Asymptotic behavior of lifetime sums for random simplicial complex processes}. \newblock{J. Math. Soc. Japan}, 71, no 3, 765--804, 2019.

\bibitem{Hoffman2017} C. Hoffman, M. Kahle and E. Paquette,  \emph{Spectral gaps of random graphs and applications}. \newblock{Int. Math. Res. Not. (to appear)}, 2019.

\bibitem{Janson11} S. Janson, T. Luczak and A. Rucinski. Random graphs. Vol. 45. John Wiley \& Sons, 2011.

\bibitem{MK} M. Kneser, \newblock\emph{Aufgabe 360}. \newblock{Jahresbericht der Deutschen Mathematiker-Vereinigung}, 58, 1955.

\bibitem{dk} D. Kozlov, \newblock{Combinatorial algebraic topology}. \newblock{Springer Verlag, Berlin} 1928, 2008.

\bibitem{kahle} M. Kahle, \newblock\emph{The neighborhood complex of a random graph}. \newblock{J. Comb. Th. Ser. A}, 114, no. 2, 380--387, 2007.

\bibitem{MK1} M. Kahle, \newblock\emph{Sharp vanishing thresholds for cohomology of random flag complexes.} 
\newblock{Ann. Math.} (2) 179, no. 3, 1085--1107, 2014. 

\bibitem{Kahle14a} M. Kahle, \emph{Topology of random simplicial complexes: a survey}. AMS Contemp. Math., 620, 201--222, 2014. 

\bibitem{lovasz} L. Lov{\'a}sz,  \newblock\emph{Kneser's conjecture, chromatic number and homotopy}. \newblock{J. Comb. Th. Ser. A}, 25, no. 3, 319-324, 1978. 

\bibitem{Munkres} J. R. Munkres, \newblock{Elements of algebraic topology}.  Addison-Wesley Publishing Company, Menlo Park, CA, 1984. 

\end{small}
\end{thebibliography}

\end{document}